\documentclass[11pt]{article} 

\usepackage{amsmath,amsthm,amssymb,geometry} 
\usepackage[applemac]{inputenc}
\usepackage[all]{xy} 
\usepackage[T1]{fontenc} 
\usepackage{textcomp}

\usepackage{mathptmx}

\DeclareMathOperator{\Hom}{Hom}

\DeclareMathOperator{\SO}{SO}

\DeclareMathOperator{\U}{U}

\DeclareMathOperator{\Ric}{Ric}

\newcommand{\R}{\mathbb R}
\newcommand{\C}{\mathbb C}
\newcommand{\Z}{\mathbb Z}

\newcommand{\diff}{\mathrm{d}}

\newcommand{\st}{\mathrm{st}}

\newcommand{\g}{\mathfrak{g}}

\renewcommand{\P}{\mathbb P}

\renewcommand{\tilde}{\widetilde}

\theoremstyle{plain}
	\newtheorem{theorem}{Theorem}
	\newtheorem{proposition}[theorem]{Proposition}
	\newtheorem{lemma}[theorem]{Lemma}
	\newtheorem{corollary}[theorem]{Corollary}
	\newtheorem{conjecture}[theorem]{Conjecture}
	\newtheorem{question}[theorem]{Question}
\theoremstyle{definition}
	\newtheorem{definition}[theorem]{Definition}

\theoremstyle{plain}
	\newtheorem*{theorem*}{Theorem}
	\newtheorem*{proposition*}{Proposition}
	\newtheorem*{lemma*}{Lemma}
	\newtheorem*{corollary*}{Corollary}
	\newtheorem*{conjecture*}{Conjecture}
\theoremstyle{definition}
	\newtheorem*{definition*}{Definition}
	\newtheorem*{remark*}{Remark}
	\newtheorem*{remarks*}{Remarks}
	
\makeatletter
\def\blfootnote{\xdef\@thefnmark{}\@footnotetext}
\makeatother

\numberwithin{theorem}{section}

  \renewenvironment{thebibliography}[1]{%
    \begin{oldthebibliography}{#1}%
      \setlength{\parskip}{.2ex}%
      \setlength{\itemsep}{.2ex}
  }%
  {%
    \end{oldthebibliography}%
  }

\begin{document}

\title{Circle-invariant fat bundles and symplectic Fano 6-manifolds}

\author{Joel Fine\footnote{JF was supported by an Action de Recherche Concertée and by an Interuniversity Action Poles grant.} 
~and
Dmitri Panov\footnote{DP is a Royal Society Research Fellow.}}

\date{\today}

\maketitle

\begin{abstract}
We prove that a compact 4-manifold which supports a circle-invariant fat $\SO(3)$-bundle is diffeomorphic to either $S^4$ or $\overline{\C\P}^2$. The proof involves studying the resulting Hamiltonian circle action on an associated symplectic 6-manifold. Applying our result to the twistor bundle of Riemannian 4-manifolds shows that $S^4$ and $\overline{\C\P}^2$ are the only 4-manifolds admitting circle-invariant metrics solving a certain curvature inequality. This can be seen as an analogue of Hsiang-Kleiner’s theorem that only $S^4$ and $\C\P^2$ admit circle-invariant metrics of positive sectional curvature.
\end{abstract}

\section{Overview}

\subsection{The main result}

To state the main result of this article, we begin with a definition.

\begin{definition}
Let $E \to M$ be an $\SO(3)$-bundle over a 4-manifold. A metric connection $A$ in $E$ is called \emph{definite} if $F_A(u,v) \neq 0$ whenever $u,v$ are linearly independent tangent vectors. 

\end{definition}

Definite connections are a special case of \emph{fat connections}, a concept introduced by Weinstein \cite{Weinstein1968Unflat-bundles,Weinstein1980Fat-bundles-and}. In \cite{Fine2009Symplectic-Cala} we rediscovered this definition in the context of twistor spaces. We feel this particular case of fat connections, namely fat $\SO(3)$-connections over 4-manifolds,  deserves a specific name because of various connections to different branches of geometry (see for example \cite{Fine2009Symplectic-Cala,Fine2013The-diversity-o,Fine2014A-gauge-theoret} for applications to minimal surfaces, construction of symplectic Calabi--Yau manifolds and the study of Einstein metrics respectively). For this reason we will keep using the name definite connection.

This article is concerned with definite connections which are invariant under a circle action. 

\begin{definition}
We say a 4-manifold $M$ admits an $S^1$-invariant definite connection if:
\begin{enumerate}
\item
There is a faithful action of $S^1$ on $M$.
\item
The action lifts to the total space of an $\SO(3)$-bundle $E\to M$, sending fibres to fibres by linear isometries.
\item
There is a definite connection in $E$ which is preserved by the $S^1$-action.
\end{enumerate}
\end{definition}

Our main result is the following.

\begin{theorem}\label{main_theorem}
If a closed 4-manifold $M$ admits an $S^1$-invariant definite connection then $M$ is diffeomorphic to either $S^4$ or $\overline{\C\P}^2$.
\end{theorem}

Here, $\overline{\C\P}^2$ denotes the complex projective plane with the opposite orientation to that induced by its complex structure. As is explained below in \S\ref{review_def_conns}, a definite connection determines an orientation on $M$ and the diffeomorphism of Theorem~\ref{main_theorem} is orientation preserving.

Before giving an outline of the proof, we first give some additional context and, in \S\ref{rel_Riem}, an application.

\subsection{Relation with symplectic geometry}\label{rel_symp}

Given a definite connection in $E \to M$, the unit sphere bundle $Z \subset E$ inherits a natural symplectic form $\omega$. This construction is explained in detail in \cite{Fine2009Symplectic-Cala} and is reviewed below in \S\ref{review_def_conns}. The symplectic manifolds which arise this way are of a very special kind. The fibres of $Z \to M$ are symplectic 2-spheres. Moreover, their normal degree $d$ (with respect to an almost complex structure tamed by $\omega$) satisfies $|d|=2$. 
There are two possibilities. When $d = +2$, we call the connection \emph{positive definite}. In this case $c_1(Z, \omega) = 2[\omega]$. When $d=-2$, we call the connection \emph{negative definite}. In this case $c_1(Z, \omega) = 0$. 

As we will see, in the setting of Theorem~\ref{main_theorem}, the $S^1$-action on $E$ restricts to a \emph{Hamiltonian} action on $Z$ which in turn implies that $A$ is necessarily a positive definite connection (Lemma~\ref{hamaction} and Corollary~\ref{positive_definite}). In general, a symplectic manifold $(X, \omega)$ for which $c_1(X) = \lambda[\omega]$ in $H^2(X, \Z)$ with $\lambda>0$ is called a \emph{symplectic Fano} (by analogy with algebraic Fano manifolds). An important question is to determine to what extent symplectic Fano manifolds differ from algebraic Fanos. 

In dimension four, a theorem of McDuff \cite{McDuff1990The-structure-o} (building on the substantial works of Gromov \cite{Gromov1985Pseudo-holomorp} and Taubes \cite{Taubes2000Seiberg--Witten}) shows that every closed 4-dimensional symplectic Fano is in fact algebraic. The corresponding question in dimension six---whether or not there exist non-algebraic symplectic Fanos---remains completely open. (In dimensions 12 and higher, \cite{Fine2010Hyperbolic-geom} shows that there are infinitely many symplectic Fano manifolds which are not algebraic.) Whilst the full six-dimensional problem seems extremely hard, there is a weaker version of this question which we would like to state in a form of conjecture.

\begin{conjecture}\label{S1_Fano_question}
Let $(Z,\omega)$ be a six-dimensional symplectic Fano manifold with a non-trivial Hamiltonian $S^1$-action. Then $Z$ is diffeomorphic to a complex algebraic Fano threefold.
\end{conjecture}

One can see Theorem~\ref{main_theorem} as a confirmation of this conjecture 
in a partial case. At the same time Theorem~\ref{main_theorem} can be 
deduced from Conjecture \ref{S1_Fano_question} using the classification 
of complex projective Fano threefolds. A result of McDuff  \cite{McDuff2009Some-6-dimensio} and Tolman \cite{Tolman2010On-a-symplectic} shows that Conjecture \ref{S1_Fano_question} holds in the case $H^2(Z,\mathbb R)=\mathbb R$. Recently we learned from Yunhyung Cho 
that from his work-in-progress it follows that the conjecture holds 
in the case when the $S^1$-action on $Z$ is semi-free. We discuss some variations of this conjecture in \S\ref{conclusion}.

\subsection{Relation with Riemannian geometry}\label{rel_Riem}

One way to produce definite connections is to start with an oriented Riemannian 4-manifold $(M,g)$ and consider the bundle $\Lambda^+ \to M$ of self-dual 2-forms. The Levi-Civita connection induces a metric connection in $\Lambda^+$ and definiteness of this connection is equivalent to a certain inequality on the curvature tensor of $(M,g)$. This inequality is explained in detail in \cite{Fine2009Symplectic-Cala}. As a consequence, Theorem~\ref{main_theorem} gives a strong restriction on Riemannian 4-manifolds with isometric $S^1$-action whose curvature satisfies this inequality. To put it precisely:

\begin{corollary}\label{Riemannian_cor}
Let $(M,g)$ be closed oriented Riemannian 4-manifold with an isometric $S^1$-action. Suppose that the curvature of $(M,g)$ satisfies the following inequality
\begin{equation}\label{Riemannian_ineq}
\left(W^+ + \frac{s}{12}\right)^2 > \Ric_0^* \Ric_0
\end{equation}
Then $M$ is diffeomorphic to $S^4$ or $\overline{\C\P}^2$, and $W^++\frac{s}{12}>0$.
\end{corollary}

Here, $s$ is the scalar curvature, $W^+ \colon \Lambda^+ \to \Lambda^+$ is the self-dual Weyl curvature and $\Ric_0$ is the trace-free Ricci curvature, suitably interpreted as a linear map $\Lambda^+ \to \Lambda^-$. The two sides of inequality~(\ref{Riemannian_ineq}) are thus self-adjoint endomorphisms of $\Lambda^+$ and the inequality asks that the difference be positive definite.

By Theorem 3.12 of \cite{Fine2009Symplectic-Cala}, if $(M,g)$ satisfies both (\ref{Riemannian_ineq}) and, in addition, $W^++s/12>0$ (but without the need to impose $S^1$-symmetry) then $M$ is \emph{homeomorphic} to a connected sum $n \overline{\C\P}^2$ where $n=0,1,2,3$. 

Inequality (\ref{Riemannian_ineq}), was analysed in \cite[Section 3]{Fine2009Symplectic-Cala}. It is explained there that
inequality (\ref{Riemannian_ineq})  defines four open components in the space 
of (pointwise) Riemann tensors. Two of these components seem to be more geometric, 
since there are compact manifolds whose Riemann tensor at each point belongs to
these components. The first geometric component 
is obtained from inequality (\ref{Riemannian_ineq})  by adding 
to it the inequality $W^++\frac{s}{12}>0$; the corresponding compact 
manifolds are $S^4$ or $\overline{\C\P}^2$. The other geometric component 
is contained in the cone or $W^++\frac{s}{12}<0$ and the compact manifolds 
are, for example, real and complex hyperbolic four manifolds. By \cite[Theorem 3.7]{Fine2009Symplectic-Cala}, 
the Riemann tensor of a 4-manifold whose sectional curvature is $\frac{5}{2}$-pinched lies in one of the two components.

It is interesting to compare Corollary \ref{Riemannian_cor} with the following theorem, whose
first part is due to Hsiang--Kleiner \cite{Hsiang1989On-the-topology} and second  is due to Grove--Wilking \cite{Grove2013A-knot-characte}.

\begin{theorem}\label{HKGW}
Let $(M,g)$ be a compact oriented 4-manifold with positive curvature and a non-trivial isometric $S^1$-action. Then 
\begin{itemize}
\item $M$ is diffeomorphic to $S^4$ or $\C\P^2$ \cite{Hsiang1989On-the-topology} and
\item This diffeomorphism can be chosen so that the $S^1$-action is linear on either $S^4$ or $\C\P^2$ \cite{Grove2013A-knot-characte}.
\end{itemize}
\end{theorem}

(In \cite{Hsiang1989On-the-topology}, Hsiang--Kleiner state the result up to homeomorphism, but their work gives a diffeomorphism when combined with Perelman's resolution of the Poincaré conjecture.)

Our results, Theorem~\ref{main_theorem} and Corollary~\ref{Riemannian_cor}, are weaker in the sense that we can not prove an analogue of Grove and Wilking's result. At the same time, our proof is almost metric-free; instead of Riemannian geometry it uses methods of symplectic geometry. Also, it is worth pointing out that a Riemann tensor that satisfies inequalities (\ref{Riemannian_ineq}) and $W^++s/12>0$ can have some negative sectional curvatures.

\subsection{Outline of the proof}

The proof of Theorem~\ref{main_theorem} relies on a theorem of Fintushel  \cite{Fintushel1978Classification-}, that the only simply connected 4-manifolds which admit non-trivial circle actions are the connected sums of copies of $S^2 \times S^2$, $\C\P^2$ and $\overline{\C\P}^2$. We will prove that when $M$ admits an $S^1$-invariant definite connection, it is simply connected and that either $\chi(M)=2$ or $\chi(M)=3$ and the signature of $M$ is $-1$. Theorem~\ref{main_theorem} then follows from Fintushel's result.

To do this we will study the fixed locus $M^{S^1}$ of the $S^1$-action as well as the larger set $M^\st$ of points with non-trivial stabiliser. By \cite[Proposition~3.1]{Fintushel1977Circle-actions-} in the case when $\pi_1(M)=0$,  the connected components of $M^\st$ are either isolated points, chains of 2-spheres or circles of 2-spheres. These components come decorated with integers from which the topology of $M$ can be determined. This background is reviewed in \S\ref{review_circle_actions}.  We call a connected component of $M^\st$ together with its decoration a \emph{pattern} in $M$. The goal then is to understand exactly what combination of patterns are possible.

The first step is to show that the induced $S^1$-action on $Z$ is Hamiltonian. As is explained in \S\ref{adjunction_section}, the spheres in the patterns lift to give symplectic spheres in $Z$ which are preserved by the Hamiltonian $S^1$-action. By studying how the Hamiltonian varies along these symplectic spheres we link the symplectic geometry of the $S^1$-action on $Z$ to the topology of the $S^1$-action on $M$. 

Our argument can be informally summarised as follows. The value of the Hamiltonian at a fixed point in $Z$ is determined by the \emph{local} geometry of the $S^1$-action at the fixed point downstairs in $M$. Meanwhile a Hamiltonian generating an $S^1$-action satisfies a strong constraint: any local maximum is necessarily a global maximum and the locus of such points is connected. This in turn gives strong constraints on the \emph{global} geometry of the $S^1$-action on $M$. With some effort  this is enough to completely determine the possible patterns and hence the diffeomorphism type of $M$. A crucial role is played here by three equations, two coming from the $G$-signature theorem on 4-manifolds and explained in Theorem~\ref{signature_constraint_theorem} and one coming from the study of definite connections, given in Proposition~\ref{adjunction_inequality}. These give additional constraints on the $S^1$-action in terms of global topological quantities. The proof works by playing these constraints off against those coming from the Hamiltonian function on $Z$.

The article is organised as follows. In \S\ref{review_circle_actions} we review the necessary parts of the theory of $S^1$-actions on 4-manifolds, the description of $M^\st$ and the result of Fintushel mentioned above. In \S\ref{review_def_conns} we review some special features of the symplectic manifolds arising from definite connections. \S\ref{first_look} proves that the $S^1$-action on $Z$ is Hamiltonian and describes how the local geometry of the action on $Z$ is determined by the local geometry of the action on $M$. In \S\ref{non-isolated} we prove the main result in the special case when there are non-isolated fixed points in $M$. In this situation it is possible to use the results of \S\ref{first_look} to give a direct proof. When all the fixed points in $M$ are isolated the argument is more involved. First, \S\ref{isolated_admissible_patterns} gives a complete characterisation of all the possible patterns. Then \S\ref{isolated_end_of_proof} examines which patterns can occur simultaneously. This gives a complete description of the possibilities for $M^\st$, ultimately proving Theorem~\ref{main_theorem}.

\subsection{Acknowledgements}

The article \cite{Wright2011Compact-anti-se} of Dominic Wright was the inspiration for some of the considerations in this article and we would like to thank him for discussing his work with us. We would also like to thank Leonor Godinho, Dusa McDuff, Anton Petrunin, Igor Rivin, Silvia Sabatini, and Wolfgang Ziller for useful discussions.

\section{Preliminaries}

\subsection{A review of $S^1$-actions on 4-manifolds}\label{review_circle_actions}

\subsubsection{Four manifolds with $S^1$-actions and relative weights}

In this section we will explain how to describe faithful $S^1$-actions on compact oriented 4-manifolds. The foundational work in this area are the articles \cite{Fintushel1977Circle-actions-,Fintushel1978Classification-} of Fintushel and \cite{Pao1978Nonlinear-circl} of Pao. The main result we need from these papers is the following. (Note that Fintushel's theorem was originally conditional on the truth of the Poincare conjecture in dimension three, subsequently proved by Perelman.)

\begin{theorem}[Fintushel \protect{\cite[Theorem 13.2]{Fintushel1978Classification-}}, Perelman]\label{Fintushel_classification}
Let $M$ be a simply-connected closed 4-manifold with a non-trivial $S^1$-action. Then $M$ is diffeomorphic to either $S^4$ or a connected sum of copies of $S^2 \times S^2$, $\C\P^2$ and $\overline{\C\P}^2$.
\end{theorem}

We will also use a description of $S^1$-actions which is very close to that employed by Fintushel (the only difference being the precise choice of integers used). We begin by assigning \emph{relative weights} to fixed points of the $S^1$-action on $M$. Let $S^1$ act by complex linear transformations on $\C^2$. Then there is an irreducible decomposition $\C^2 = L_1 \oplus L_2$ into lines and a pair of integers $a, b$---called the \emph{weights} of the action---such that  $e^{i\theta} \in S^1$ acts as multiplication by $e^{ia\theta}$ on $L_1$ and by $e^{ib\theta}$ on $L_2$. 

Now consider a non-trivial real-linear action of $S^1$ on $\R^4$ with a fixed orientation. There exist exactly two $S^1$-invariant linear complex structures $J$ and $-J$ on $\R^4$, compatible with the orientation. Hence we can talk of the weights of the $S^1$-action as in the previous paragraph, but the weights are defined only up to an overall sign. 

\begin{definition}
Fix an orientation on $\R^4$ and let $S^1$ act linearly on $\R^4$. Given an $S^1$-invariant linear almost complex structure $J$ on $\R^4$ which is compatible with the orientation, we define the \emph{relative weights} of the action to be the weights $a,b$ of the complex-linear $S^1$-action  on $(\R^4, J)$ which are uniquely determined up to overall choice of sign. 

Now let $S^1$ act on an oriented 4-manifold. We define the \emph{relative weights} of a fixed point $p$ to be those of the induced $S^1$-action on $T_pM$.
\end{definition}

\begin{lemma}\label{coprime}
If $S^1$ acts faithfully on a 4-manifold $M$, then the relative weights of any fixed point are coprime. 
\end{lemma}
\begin{proof}
If the relative weights of a fixed point $p$ are not coprime there is a non-trivial element $e^{i\theta} \in S^1$ which fixes both $p$ and $T_pM$. This means that $e^{i\theta}$ acts trivially on the whole of $M$ (e.g., pick an $S^1$-invariant metric on $M$, then $e^{i\theta}$ is an isometry fixing a point and its tangent space and so must be the identity). It follows that the action is not faithful.
\end{proof}

\subsubsection{Patterns of $S^1$-actions on 4-manifolds}

We now pass from the set of fixed points $M^{S^1}$, to the larger set $M^\st$ of points with non-trivial stabiliser. The connected components of $M^\st$ will occur frequently in our proof of Theorem~\ref{main_theorem} and so we give them a name.

\begin{definition}
A connected component of $M^\st$, together with the relative weights of each fixed point in the component, is called a \emph{pattern}.
\end{definition}

In what follows we assume $\pi_1(M)=0$.  
(We will prove later on that 4-manifolds with $S^1$-invariant definite connections are simply connected.) We give a description of the possible patterns. This is contained in \cite[Proposition 3.1]{Fintushel1978Classification-} but formulated using different terminology.

Let $p \in M^\st$ have finite stabiliser $\Z_m \subset S^1$. Then (when $\pi_1(M)=0$) $p$ lies on a 2-sphere $S \subset M^\st$ which contains exactly two fixed points $p_0, p_1 \in S \cap M^{S^1}$; the remaining points of $S\setminus\{p_0, p_1\}$ all have the same stabiliser $\Z_m$ as $p$; moreover, $p_0, p_1$ each have $\pm m$ as one of their relative weights and $T_{p_j}S$ is the corresponding 2-plane in $T_{p_j}M$ which $S^1$ rotates with speed $m$.

Conversely, if $p$ is a fixed point with relative weights $(a,b)$ neither of which have modulus one, then there are exactly two distinguished 2-planes $L_1, L_2 \subset T_{p}$ which are rotated with speeds $|a|, |b|$ by the $S^1$-action (here we use that neither of the weights are equal to one and that they are coprime). Moreover, there are a pair of 2-spheres $S_1, S_2 \subset M^\st$ which pass through $p$ with $T_pS_j = L_j$ and whose generic points have stabiliser $\Z_{|a|}$ and $\Z_{|b|}$ respectively.

If $p$ is a fixed point with relative weights $(a, b)$ with $|a|=1$ and $|b| >1$ then there is a single 2-sphere in $M^{\st}$ which ends at $p$. Finally, if $p$ is a fixed point with relative weights satisfying $|a|=1=|b|$, then $p$ is isolated in $M^\st$. 

This gives the following four possible types of pattern:
\begin{enumerate}
\item
An \emph{isolated point  of $M^{\st}$} with relative weights $(1,1)$ or $(1,-1)$. 
\item
An \emph{arc of 2-spheres}, i.e., a connected union of spheres from $M^{\st}$ that projects to a topological arc in the 3-manifold $M^* = M/S^1$. Every sphere in the arc contains exactly two points from $M^{S^1}$. Both of the endpoints of the arc in $M^*$ have preimage a fixed point with relative weights $(a,b)$ where either $|a|=1$ or $|b|=1$ (but not both). Note that an arc can be made up of a single sphere. 
\item
A \emph{circle of 2-spheres}, i.e., a connected union of spheres from $M^{\st}$ that projects to a topological circle in $M^*$. For each sphere in the circle both fixed points have relative weights $(a,b)$ with $|a|, |b|>1$. Note that a circle of spheres can not be composed of 
just one sphere. Indeed, 
the $S^1$-fixed point of such a sphere would 
have relative weights $(a,\pm a)$ with $|a|>1$,
which contradicts Lemma \ref{coprime}. 
\item
An entire surface $\Sigma$ fixed pointwise by $S^1$. Such surfaces project isomorphically to a boundary component of $M^*$.
\end{enumerate}

\subsubsection{A numerical constraint and the signature formula}

Not every collection of patterns can occur via an $S^1$-action on a 4-manifold $M$; there is an important numerical constraint which must be satisfied. Moreover, one can also recover the signature directly from the patterns. These two facts are explained in the following theorem. (In \cite{Fintushel1977Circle-actions-,Fintushel1978Classification-}, Fintushel uses different integers than the relative weights to decorate his patterns with and he also obtains a numerical constraint which his integers must satisfy. This is presumably ultimately the same as that of equation (\ref{constraint}) below, although we have not verified this.)

\begin{theorem}\label{signature_constraint_theorem}
Let $S^1$ act faithfully on a closed oriented 4-manifold. Suppose that the action has isolated fixed points $\{p_j\}$ with relative weights $(a_j,b_j)$ and fixed surfaces $\Sigma_k$. Then
\begin{eqnarray}
-\sum_j \frac{1}{a_jb_j} + \sum_k [\Sigma_k]\cdot[\Sigma_k]
	&=&
		0 \label{constraint}\\
\sum_j \frac{a_j^2 + b_j^2}{a_jb_j} + \sum_k [\Sigma_k] \cdot [\Sigma_k]
	&=&
		3 \sigma(M) \label{signature}
\end{eqnarray}
where $\sigma(M)$ is the signature of $M$.
\end{theorem}

\begin{proof} The second equation follows from a theorem of Bott \cite{Bott1967Vector-fields-a},
which expresses all Pontryagin numbers of a manifold $M$ with an $S^1$-action it terms of the topology of $M^{S^1}$ and the  weights of $S^1$-action 
on the normal bundle to $M^{S^1}$. In particular in case then $\dim(M)=4$ one gets a formula for $p_1(M)=3\sigma(M)$. See also Searle--Yang \cite{Searle1994On-the-topology}.

Both equations can be deduced simultaneously from the equivariant signature theorem of Atiyah--Singer \cite{Atiyah1968The-index-of-el}. For almost all rational numbers $t=m/n$, the equivariant signature theorem applied to the action of $e^{2\pi t i} \in S^1$ gives 
\begin{equation}\label{signature_functional_eq}
\sigma(M) 
= 
- \sum_j \frac{\cos(\pi a_j t) \cos(\pi b_j t)}{\sin (\pi a_j t) \sin(\pi b_j t)}
+
\sum_k \frac{1}{\sin^2(\pi t)} [\Sigma_k] \cdot [\Sigma_k]
\end{equation}
(See equation (18) in the article \cite{Hirzebruch1971The-signature-t} of Hirzebruch; one needs that the only fixed points of $e^{2\pi t i}$ are fixed by the whole of $S^1$.) The right-hand-side of (\ref{signature_functional_eq}) can be written as a Laurent series at $t=0$. In this expansion, the coefficient of $t^{-2}$ must be zero and the constant term equal to $\sigma(M)$, giving the two equations (\ref{constraint}) and~(\ref{signature}).
\end{proof}

\subsection{A review of the symplectic geometry of definite connections}\label{review_def_conns}

\subsubsection{From connections over 4-manifolds to symplectic 6-manifolds}

We now review how a definite connection $A$ in an $\SO(3)$-bundle $E \to M$ over a 4-manifold gives rise to a symplectic structure $\omega_A$ on the total space of the unit sphere bundle $Z \subset E$. The details of this construction can be found in \cite{Fine2009Symplectic-Cala}. 

We first explain how \emph{any} $\SO(3)$-connection $A$ naturally defines a closed 2-form $\omega_A$ on $Z$. Write $V\to Z$ for the vertical tangent bundle of $Z$, i.e., the sub-bundle of $TZ$ containing vectors tangent to the fibres of $Z \to M$. $V$ is an $\SO(2)$-bundle; once an orientation of the fibres of $E$ is fixed, $V$ can be regarded as a $\U(1)$-bundle. The closed 2-form $\omega_A$ is the curvature of a unitary connection $\nabla$ in~$V$. Let $s$ be a section of $V$ and $u$ tangent to $Z$. When $u$ is itself vertical, $\nabla_u s$ is simply given by the Levi-Civita connection on the tangent bundle of the fibre. Meanwhile, the connection $A$ defines a vertical-horizontal splitting $TZ = V \oplus H$. When $u$ is horizontal, parallel transport with respect to $A$ along the projection of $u$ to $M$ identifies nearby fibres of $Z \to M$ and so also their tangent spaces. $\nabla_u s$ is then the ordinary derivative under this identification. We now set $\omega_A = \frac{1}{2\pi i} F_\nabla$, a closed 2-form which restricts to each fibre to give its area form. One can check that the $\omega_A$-complement of $V$ is again $H$, the horizontal distribution of $A$. Moreover the restriction of $\omega_A$ to $H$ can be described in terms of the curvature of $A$. For details see \cite[\S2.1]{Fine2009Symplectic-Cala}.

The main interest in $\omega_A$ is when it is in fact \emph{symplectic}. This corresponds precisely to $A$ being a definite connection:

\begin{proposition}[\protect{\cite[\S2.2]{Fine2009Symplectic-Cala}}]
$(Z, \omega_A)$ is symplectic if and only if $A$ is a definite connection. 
\end{proposition}

An important point is that a definite connection $A$ induces a preferred orientation on $M$. This is because the sub-bundle $H \subset Z$ is symplectic, hence oriented by $\omega_A^2$ and this orientation descends to $TM$. Note that this is independent of the choice of orientation of the fibres of $E$; the opposite choice leads to the 2-form $-\omega_A$ and so does not affect $\omega_A^2$. Now the orientation of $M$ determines in turn a preferred orientation of the fibres of $E$: we declare them to be positively oriented if the push-forward of $\omega_A^3$ is a positive 4-form on $M$. From now on, given a definite connection $A$ in $E$ we will use these induced orientations on $M$ and on the fibres of $E$.

The symplectic manifolds $(Z, \omega_A)$ which arise this way are of a very special sort, either ``symplectic Fano'' or ``symplectic Calabi--Yau'':

\begin{proposition}[\protect{\cite[\S2.3]{Fine2009Symplectic-Cala}}]
When $A$ is a definite connection, the fibres of $Z \to M$ are symplectic 2-spheres whose normal degree $d$ satisfies $|d|=2$.
\begin{itemize}
\item
If $d = 2$, we say $A$ is \emph{positive definite}. In this case, $c_1(Z, \omega_A) = 2[\omega_A]$.
\item
If $d = -2$, we say $A$ is \emph{negative definite}. In this case, $c_1(Z, \omega_A) = 0$.
\end{itemize}
\end{proposition}
(Here the normal degree is defined with respect to an almost complex structure tamed by $\omega$; the result does not depend on the choice, which lies in a contractible set.)

It is an interesting question to ask which 4-manifolds admit definite connections. To date, there is only one known obstruction due to Derdzinski and Rigas, an inequality involving the Euler characteristic $\chi$ and the signature  $\sigma$ (defined with respect to the orientation on $M$ induced by $A$). 

\begin{proposition}[\cite{Derdzinski1981Unflat-connecti}, see also \protect{\cite[\S2.3]{Fine2009Symplectic-Cala}}] 
Let $M$ be a compact 4-manifold with a definite connection. Then $2\chi(M)+3\sigma(M)>0$.
\end{proposition}

\subsubsection{An adjunction inequality}\label{adjunction_section}

Our proof of Theorem~\ref{main_theorem} is based on lifting the spheres from $M^\st$ to the symplectic manifold~$Z$. We now describe how to do this. Let $f \colon \Sigma \to M$ be an embedding of an oriented surface. There is natural lift $\tilde{f} \colon \Sigma \to Z$ defined as follows. Let $u,v$ be an oriented basis for $T_\sigma \Sigma$. Since $A$ is definite, the curvature $F_A(f_*u, f_*v)$ generates a non-zero rotation of the fibre $E_{f(\sigma)}$ of $E$. Moreover, since $E$ is oriented, there is a uniquely determined unit-length $p \in E_{f(\sigma)}$ such that $F_A(f_*u, f_*v)$ is a positive multiple of the cross-product with $p$. Setting $\tilde{f}(\sigma) = p$ defines the lift of $f$ to $Z$.

A key fact about these lifts is that their symplectic area is determined entirely by the topology of $\Sigma$ and $f$:

\begin{proposition}[\protect{\cite[\S4.4]{Fine2009Symplectic-Cala}}]\label{area_is_adjunction}
Let $f\colon \Sigma \to M$ be an embedding of an oriented surface. Then
\[
\int_\Sigma \tilde{f}^*\omega_A  = \pm 2\pi\left( \chi(\Sigma) + [f(\Sigma)]\cdot [f(\Sigma)]\right)\
\]
where the sign agrees with that of the definite connection. 
\end{proposition}

This has the following important corollary, which is a direct analogue of the adjunction inequality of \cite[\S4.4]{Fine2009Symplectic-Cala}.

\begin{proposition}\label{adjunction_inequality} Let $E \to M$ be an $\SO(3)$-bundle and $A$ a positive definite connection in $E$. Suppose that $\Z_n$ acts on $E$ by fibrewise linear isometrics and preserves $A$. If $\Sigma \subset M$ is an orientable surface fixed pointwise by $\mathbb Z_n$, then
\[
\chi(\Sigma) + [\Sigma]\cdot [\Sigma] > 0
\]
Similarly, if $A$ is negative definite, this quantity is negative.
\end{proposition}

\begin{proof}
The $\Z_n$-action on $E$ preserves the unit sphere-bundle $Z$ where is acts by symplectomorphisms. Now the lifts $\Sigma_1, \Sigma_2 \subset Z$ of $\Sigma$ (one for each orientation) are also fixed pointwise by $\Z_n$ and hence are symplectic (because they are components of the fixed set of a symplectomorphism of finite order). The result now follows from Proposition~\ref{area_is_adjunction}.
\end{proof}

\section{The $S^1$-action on $Z$}\label{first_look}

\subsection{The action is Hamiltonian}

With the preliminaries in hand, we now begin the proof of Theorem~\ref{main_theorem}. The circle action on $E$ preserves the unit sphere bundle $Z$ where it acts by symplectomorphicms. The first step is to show that this action is Hamiltonian. This follows from a general result about actions which lift to line bundles. This is well-known, but we give a proof for completeness (and lack of an explicit reference).

\begin{lemma}\label{hamaction}
Let $\pi\colon L \to Z$ be a Hermitian line bundle and $B$ a unitary connection such that $\frac{i}{2\pi}F_B = \omega$ is a symplectic form. Suppose that $S^1$ acts on $L$ by fibrewise linear isometries and preserves $A$. Then the induced $S^1$ action on $Z$ is Hamiltonian.
\end{lemma}

\begin{proof}
Let $w$ be the vector field on $L$ generating the $S^1$-action. Then $v=\pi_*(w)$ is the vector field on $Z$ generating the $S^1$-action. Write $\hat{v}$ for the horizontal lift of $v$ to $L$ with respect to $B$. Since $w$ covers $v$, the horizontal component of $w$ is $\hat{v}$. Meanwhile, since $w$ generates fibrewise linear isometries, the vertical component is given by multiplication on each fibre by an imaginary number. Accordingly, we write
\[
w = \hat{v} + 2\pi i f
\]
for some function $f \colon Z \to \R$.

Restricting to the unit circle bundle $P \subset L$ we can think of the connection $B$ as a 1-form with values in imaginary numbers. I.e., $P$ is the principal frame bundle of $L$ and $B$ is a connection 1-form. The flow of $w$ preserves $P$ and, moreover $B$. So $0= L_wB = \diff (B(w)) + \iota_w(\diff B)$. Now, on $P$, the curvature and connection 1-form are related by $\diff B = \pi^* F_B$. Moreover, $B(w) = 2\pi i f$. Hence
\[
2\pi i \diff f  + \iota_v(F_B) = 0
\]
In other words, $\diff f = \iota_v \omega$ and so $v$ is Hamiltonian as claimed.
\end{proof}

\begin{corollary}\label{Z_action_Hamiltonian}
Let $M$ admit an $S^1$-invariant definite connection in an $\SO(3)$-bundle $E \to M$. The restriction of the $S^1$-action to the unit sphere bundle $(Z, \omega_A)$ is Hamiltonian.
\end{corollary}

\begin{proof}
This follows from the description of $\omega_A$ as the curvature of the unitary connection $\nabla$ in the Hermitian line bundle $V \to Z$, outlined in \S\ref{review_def_conns}. The $S^1$-action lifts from $Z$ to the $V$ where it preserves $\nabla$, implying the action is Hamiltonian. 
\end{proof}

\begin{corollary}\label{positive_definite}
If $M$ admits an $S^1$-invariant definite connection it must be  \emph{positive} definite.
\end{corollary}
\begin{proof}
If the connection were negative definite, $Z$ would be a compact symplectic Calabi-Yau manifold. But by \cite{Cho2012Chern-classes-a}, such manifolds do not admit Hamiltonian $S^1$-actions.
\end{proof}

The following two results about Hamiltonian circle actions will be crucial in what follows. 

\begin{lemma}[See \protect{\cite[Lemma 5.51]{McDuff1998Introduction-to}}]
\label{minmax}
Let $H$ be a Hamiltonian on a symplectic manifold $(Z,\omega)$ that generates an $S^1$-action. Then all level sets of $H$ are connected. In particular any local maximum or minimum of $H$ is actually a global maximum or minimum and the two sets $Z_{\max}$  and $Z_{\min}$ where $H$ attains its maximum or minimum respectively are connected.
\end{lemma}
 
\begin{theorem}[Hui Li, \cite{Li2006The-fundamental}]\label{huili}  Let $(Z, \omega)$ be a connected, compact, symplectic manifold with a Hamiltonian $S^1$-action. Then the following natural homomorphisms, induced by inclusions, are isomorphisms.
$$\pi_1(Z_{\min})\to \pi_1(Z),\quad \pi_1(Z_{\max})\to \pi_1(Z).$$
\end{theorem}

\subsection{A first look at the $S^1$-action on $Z$}

We conclude this section by describing in a little more detail the action of $S^1$ on $Z$. Our ultimate goal here is to relate the action on $Z$ to that on $M$ as well as to determine the value of the Hamiltonian function at fixed points of $Z$ in terms of the relative weights downstairs. To do this, we first give a natural choice of normalisation for the Hamiltonian, exploiting the involution $\gamma \colon Z \to Z$ of the unit sphere bundle given by multiplication by $-1$ in the fibres. 

\begin{lemma}
The Hamiltonian function $H \colon Z \to Z$ generating the $S^1$-action on $(Z, \omega_A)$ can be chosen so that $H(\gamma z) =  - H(z)$ for all $z \in Z$. This uniquely determines $H$.
\end{lemma}
\begin{proof}
The involution $\gamma$ is anti-symplectic: $\gamma^*\omega_A = -\omega_A$. (This follows immediately from the description of $\omega_A$ given in \S\ref{review_def_conns}.) Meanwhile, the $S^1$-action is linear and so commutes with $\gamma$. It follows that if $v$ is the Hamiltonian vector field of $H$ then $H \circ \gamma$ is a Hamiltonian for $-v$. Hence $H+H \circ \gamma$ is a constant which can be taken to be zero by adding a constant to $H$. 
\end{proof}

\noindent From now on, we will exclusively use this normalisation for the Hamiltonian function $H$. 

We next turn to the weights of the $S^1$-action on $Z$. Let $z \in Z$ be a fixed point. Since the action preserves the symplectic form $\omega_A$, we can define \emph{genuine} weights for the action at $z$, as opposed to ones defined only up to a common sign as was the case for fixed points in the 4-manifold. Moreover, the action on $Z$ preserves the connection $A$ and hence the splitting $TZ= V \oplus H$. It makes sense then to talk of horizontal and vertical weights of the action at $z$ as the weights of the $S^1$-action on $H_z$ and $V_z$ respectively.

\begin{lemma}\label{vertical_weights}
Let $z \in Z$ be fixed by the $S^1$-action, with horizontal weights $a,b$. Then the vertical weight $w$ at $z$ is given by $w = a+b$.
\end{lemma}
\begin{proof}
To compute the vertical weight it is convenient to use a particular choice of $S^1$-invariant almost complex structure $J$ on $Z$ defined in \cite[Definition~2.9]{Fine2009Symplectic-Cala}. It is proved there that with this choice of $J$, the splitting $TZ = V \oplus H$ is complex linear and, moreover, $V \cong \Lambda^2 H$. It is also clear from the definition that this isomorphism is $S^1$-equivariant. From here the claim that $w = a+b$ is immediate. (Note this uses that $A$ is a \emph{positive} definite connection, as is guaranteed by Corollary~\ref{positive_definite}; for negative definite connections one has $V^* \cong \Lambda^2 H$.) 
\end{proof}

With this Lemma in hand we can now explain how the action on $Z$ relates to that on $M$. Let $p \in M$ be fixed by the $S^1$-action, with relative weights $(a,b)$. $S^1$ acts isometrically on the sphere $Z_p$ over $p$ and there are two possibilities: either it fixes it completely, or it rotates it around two antipodal fixed points. Given a fixed point $z \in Z_p$, the action on $H_z$ is isomorphic to that on $T_pM$. It follows that the horizontal weights at $z$ must be equal, up to sign, to the relative weights of $p$. If $a+b \neq 0$, this means that the sphere $Z_p$ is not fixed, but rotated with speed $|a + b|$ about two fixed points $z, \gamma(z)$ with weights $a,b,a+b$ and $-a,-b,-a-b$ respectively. On the other hand, if $a+b=0$ (and so $a=1, b=-1$ or vice versa) the whole sphere $Z_p$ is fixed pointwise and every point in $Z_p$ has horizontal weights $1,-1$ and vertical weight $0$. 

We can already give one simple but very important consequence of this description.

\begin{lemma}\label{relative_weights_equal_sign}
Let $p \in M$ be a fixed point of the $S^1$-action with relative weights of the same sign. Then $H$ attains both a global maximum and minimum at fixed points in the fibre of $Z\to M$ over $p$. 
\end{lemma}
\begin{proof}
$H$ attains a local maximum at a fixed point $z$ precisely when all weights there are positive. This happens for some $z$ lying above $p$ by Lemma~\ref{vertical_weights} and the fact that the horizontal weights of fixed points above $p$ equal, up to sign, those of the relative weights downstairs. Since the level sets of $H$ are connected, any local maximum is in fact a global maximum. A similar argument, or the fact that $H \circ \gamma = - H$, proves the analogous statement about the global minimum.
\end{proof}

\begin{corollary}\label{connected_max}
The set of fixed points in $M$ with relative weights of the same sign is connected. 
\end{corollary}
\begin{proof}
We have seen that the set of such fixed points is the projection of those points in $Z$ where $H$ attains its maximum. The result follows from the fact that the maximal level set of $H$ is connected.
\end{proof}

We close this section by giving the value of the normalised Hamiltonian at fixed points of $Z$. 

\begin{lemma}\label{H_via_relative_weights}
Let $z \in M$ be fixed by the $S^1$-action with horizontal weights $(a,b)$. Then $H(z) = a+b$.
\end{lemma}
\begin{proof}
We use Archimedes theorem, that if $h$ is a Hamiltonian function generating an $S^1$-action on $S^2$, which has symplectic area $4\pi c$ and which is rotated $m$ times, then $h_{\max} - h_{\min} = 2cm$. By construction, $\omega_A$ gives all the fibres of $Z \to M$ the area $4\pi$. By Lemma~\ref{vertical_weights}, the sphere through $z, \gamma(z)$ is rotated $|a+b|$~times by the $S^1$-action. It follows that $|H(z) - H(\gamma(z))| = 2|a+b|$. Since $H(z) = - H(\gamma(z))$ this gives $|H(z)| = |a+b|$. Now if $H(z)>0$ then $H(\gamma(z))<0$, $H$ is decreasing down the vertical sphere through $z$ and so the vertical weight at $z$ is positive
\footnote{In our convention $\diff H=\omega(X_H,\cdot)$, where $X_H$ is the Hamiltonian vector field.}. By Lemma~\ref{vertical_weights}, this weight is $a+b$ and so $H(z) = a+b$. Similarly, if $H(z)<0$, the vertical weight $a+b$ is also negative and again $H(z) = a+b$.
\end{proof}

\section{The case of non-isolated fixed points}\label{non-isolated}

The goal of this section is to prove Theorem~\ref{main_theorem} in the special case where there are non-isolated fixed points.

\begin{theorem}\label{non-isolated_fixed_points}
Let $M$ be a closed 4-manifold admitting an $S^1$-invariant definite connection . If the $S^1$-action on $M$ has non-isolated fixed points then $M$ is diffeomorphic to $S^4$ or $\overline{\C\P}^2$.
\end{theorem}

The proof proceeds by a series of Lemmas. 

\begin{lemma}\label{Sigma_fixed_1}
Let $M$ be a closed 4-manifold which admits an $S^1$-invariant definite connection. Suppose that the $S^1$-action on $M$ has non-isolated fixed points. Then the non-isolated fixed points form a connected surface and the only other fixed points are isolated with relative weights $(1,-1)$. 
\end{lemma}

\begin{proof}
Write $\Sigma$ for the locus of non-isolated fixed points. By Lemma~\ref{coprime}, each point of $\Sigma$ has relative weights $(1,0)$. Corollary~\ref{connected_max} gives that $\Sigma$ is connected. Moreover, by Lemma~\ref{H_via_relative_weights}, the value of $H$ at fixed points above $\Sigma$ satisfies $|H(z)|=1$. So $-1 \leq H \leq 1$. Now if $q \in M$ is any isolated fixed point with relative weights $(a,b)$ there are fixed points $z_\pm \in Z$ above $q$ with $H(z_+) = a+b$ and $H(z_-) = -(a+b)$. Since $-1 \leq H \leq 1$ and neither $a$ nor $b$ vanishes, it follows that the relative weights are $(1,-1)$ as claimed. 
\end{proof}

\begin{lemma}\label{sphere_plus_point}
Let $M$ be a closed 4-manifold which admits an $S^1$-invariant definite connection. Suppose that the $S^1$-action on $M$ has non-isolated fixed points. Then the surface of non-isolated fixed points is a 2-sphere and there is at most one isolated fixed point.
\end{lemma}
\begin{proof}
Write $n$ for the number of isolated fixed points. By Lemma~\ref{Sigma_fixed_1} any isolated fixed point has relative weights $(1,-1)$. As above, write $\Sigma$ for the locus of non-isolated fixed points, which we also know is connected. Recall equation (\ref{constraint}) which gives a constraint on the relative weights of isolated fixed points. In the situation at hand it reads
\begin{equation}\label{self_intersection}
[\Sigma] \cdot [\Sigma] =-n \leq 0
\end{equation}
Meanwhile, $\Sigma$ is fixed by the $S^1$-action and so the adjunction inequality of Proposition~\ref{adjunction_inequality} reads
\[
2g -2 < [\Sigma] \cdot [\Sigma] 
\]
where $g$ is the genus of $\Sigma$. (We use here also the fact that the definite connection is positive, Corollary~\ref{positive_definite}.) So the only possibility is that $g=0$ and $n < 2$ as claimed.
\end{proof}

\begin{proof}[Proof of Theorem~\ref{non-isolated_fixed_points}] We assume that the $S^1$-action has non-isolated fixed points and we want to show that $M$ is diffeomorphic to either $S^4$ or $\overline{\C\P}^2$. 

We first show that $\pi_1(M) = 1$. Since $Z \to M$ is an $S^2$-bundle, $\pi_1(M) \cong \pi_1(Z)$. By Lemma~\ref{sphere_plus_point}, the non-isolated fixed points form a sphere $\Sigma \subset M$. Moreover, we have seen that $Z_{\max}$, the locus where $H$ attains its maximal value, projects diffeomorphically onto $\Sigma$. So $\pi_1(Z_{\max}) = 1$. But by Theorem~\ref{huili}, $\pi_1(Z) \cong \pi_1(Z_{\max})$ and hence $M$ is simply connected. 

We now compute the Euler characteristic. By Lemma~\ref{sphere_plus_point}, the zero locus of the vector field generating the $S^1$-action is a single 2-sphere together with at most one other isolated zero. The 2-sphere contributes 2 to $\chi(M)$. If there is no other fixed point we see that $\chi(M)=2$ and, looking at Fintushel's classification Theorem~\ref{Fintushel_classification}, we conclude that $M$ is diffeomorphic to $S^4$. 

Meanwhile, if there is another isolated fixed point, it contributes 1 to the Euler characteristic giving $\chi(M)=3$. Next we compute the signature, via equation (\ref{signature}). Equation (\ref{self_intersection}) says that in this case the 2-sphere of non-isolated fixed points has self-intersection $-1$ whilst the isolated fixed point has relative weights $(1,-1)$. Equation (\ref{signature}) then gives $\sigma(M) = -1$. Again looking at Fintushel's classification we conclude that $M$ is diffeomorphic to $\overline{\C\P}^2$.
\end{proof}

\section{All fixed points are isolated: the admissible patterns}\label{isolated_admissible_patterns}

We now turn to the case when all fixed points in the 4-manifold $M$ are isolated. This involves more work than in the previous section. The first step in the argument is to identify all the patterns which can occur. (Recall that a \emph{pattern} is a connected component of $M^\st$, the set of points with non-trivial stabiliser, together with the relative weights of the fixed points.)

\begin{theorem}\label{admiss} Let $M$ be a closed $4$-manifold
admitting an $S^1$-invariant definite connection. The following list contains all the possible patterns on~$M$:
\begin{enumerate}
	\item
	A fixed point with relative weights $(1,1)$ or $(1,-1)$.
	\item
	An arc of spheres of length one or two. 
	\item
	A circle of spheres of length two or three.
	\item
	An $S^2$ of non-isolated fixed points.
\end{enumerate}	
\end{theorem}

The previous section already showed that if there are non-isolated fixed points, they must lie on a single 2-sphere. The point of Theorem~\ref{admiss} is to rule out arcs of more than 2 spheres and circles of more than three spheres. 

The key idea is to analyse the behaviour of the Hamiltonian function $H$ on lifts of the arcs and spheres. Recall the discussion from \S\ref{adjunction_section} which explained how each embedded surface $\Sigma$ has two natural lifts to $Z$. When $\Sigma \subset M^\st$ is a 2-sphere whose generic points have stabiliser $\Z_m \subset S^1$, then the generic points of a lift $\Sigma'$ also have stabiliser $\Z_m$ This means the $S^1$-action rotates $\Sigma'$ $m$-times and we can thus compute the change in the Hamiltonian along $\Sigma'$ via its symplectic area.

\begin{lemma}\label{horizontal_change_in_H}
Write $z_0, z_1 \in \Sigma'$ for the fixed points of the $S^1$-action. 
$|H(z_0)-H(z_1)| = (2+s)m$ where $s = [\Sigma]\cdot[\Sigma]$ is the self-intersection of $\Sigma$ in $M$.
\end{lemma}
\begin{proof}
We use Archimedes theorem, that if $h$ is a Hamiltonian function generating an $S^1$-action on $S^2$, which has symplectic area $4\pi c$ and which is rotated $m$ times, then $h_{\max} - h_{\min} = 2cm$. The result follows from Proposition~\ref{area_is_adjunction}, which gives the area of $\Sigma'$ as $2\pi(2+s)$.
\end{proof}

\begin{figure}
\[
\xymatrix@R=2pt@C=4pt{
& (-a,-m) 
	&&\Sigma'&& (m,b) &\\
&\bullet \ar@{-}[rrrr] \ar@{-}[ddddd]\ar@{.}[l]
	&&&& \bullet \ar@{-}[ddddd]\ar@{.}[r] &\\
&z_0\ \ \ \ \ \ &&&&\ \ \ \ \ \ z_1&\\
&&&&&&\\
&&&&&&\\
&&&&&&\\
&\bullet \ar@{-}[rrrr]\ar@{.}[l]
	&&&& \bullet \ar@{.}[r] &\\
& (a,m)
	&&&& (-m,-b) &\\
&&&&&&\\
&p_0&&&&p_1&\\
&\bullet \ar@{-}[rrrr]\ar@{.}[l]
	&&&& \bullet \ar@{.}[r] &\\
& (a,m)
	&&\Sigma&& (m,b) &
}
\]
\caption{At the bottom, two fixed points $p_0, p_1 \in M$ are shown with their relative weights. They lie on a 2-sphere $\Sigma \subset M^\st$. Above, in $Z$, the four fixed points are shown with their horizontal weights. The points fixed points $z_0, z_1$ lie on the same lift $\Sigma'$ of $\Sigma$. It follows that the horizontal weight at $z_1$ in the direction of $\Sigma'$ is minus that at $z_0$.}
\label{curve_diagram}
\end{figure}
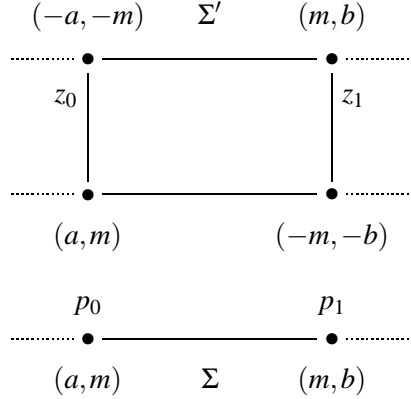

\subsection{Arcs of spheres in $M^\st$}

\begin{proposition}\label{arc_length_at_most_2}
Any arc of spheres in $M^\st$ contains at most two spheres.
\end{proposition}

We will prove this result in a series of steps. Each arc must end at fixed points with relative weights $(\pm 1,m)$ where $m > 1$. We work inwards from the end of the arc, using Lemmas~\ref{H_via_relative_weights} and~\ref{horizontal_change_in_H} to compute how $H$ varies along the lift to $Z$ of an arc.
 
\begin{lemma}\label{no_max_at_start}
Let $p \in M$ be a fixed point with relative weights $(a,m)$ where $|a|=1$ and $m >1$ and let $\Sigma \subset M^\st$ be the 2-sphere containing $p$. Let $z_0, z_1$ be the fixed points on one of the lifts $\Sigma'$ of $\Sigma$, with $z_0$ projecting to $p$. If neither $z_0$ nor $z_1$ are the maximum or minimum of $H$, then $\Sigma$ has self-intersection $-1$ and, moreover, $|H(z_1)| = 1$.
\end{lemma}
\begin{proof}
Since $z_0$ is neither a maximum nor minimum, Lemma~\ref{relative_weights_equal_sign} tells that the relative weights of $p$ have different signs. So the horizontal weights of the action at $z_0$ are either $(1,-m)$ or $(-1, m)$. We will consider the case $(1,-m)$ and prove that $H(z_1)=1$. This is the situation of Figure~\ref{curve_diagram}, with $a=-1$. The case $(-1,m)$ leads similarly to $H(z_1)=-1$. 

By Lemma~\ref{H_via_relative_weights}, $H(z_0) = 1 - m$. We next compute $H(z_1)$. Since $z_0$ and $z_1$ are the fixed points of $\Sigma'$, one of the horizontal weights at $z_1$ is $m$, i.e., minus the corresponding weight at $z_0$. Since $z_1$ is not a maximum, the other horizontal weight $b$ must be negative. (If the second horizontal weight were also positive, Lemma~\ref{relative_weights_equal_sign} would imply $z_1$ is a local maximum.) By Lemma~\ref{H_via_relative_weights}, $H(z_1) = m+b$. 

Since the horizontal weight at $z_1$ in the direction of $z_0$ is positive, $H(z_1) > H(z_0)$. So Lemma~\ref{horizontal_change_in_H} gives
\[
(m+b) - (1-m) = (2+s)m
\] 
where $s$ is the self-intersection of $\Sigma$. This simplifies to $sm = b -1$. Recalling that $b<0$ and $m>0$ we see that $s<0$. Since, by Proposition~\ref{adjunction_inequality}, $s>-2$, the only possibility is that $s=-1$ and $m=1-b$ which in turn gives $H(z_1)=1$
\end{proof}

Let $p_0 \in M$ be a fixed point with relative weights $(\pm 1,m)$ where $m >1$ and let $\Sigma_1 \subset M^\st$ be the 2-sphere containing $p_0$. Assume that $\Sigma_1$ is part of an arc of more than one sphere. Then the other fixed point $p_1\in \Sigma_1$ has relative weights $(m,n)$ where $|n|>1$. Now $p_1$ lies on a second 2-sphere $\Sigma_2 \subset M^\st$ and we write $p_2 \in \Sigma_2$ for the other fixed point there. Next we choose lifts $\Sigma_j'\subset Z$ of the $\Sigma_j$ such that $\Sigma'_1 \cap \Sigma_2' \neq \varnothing$. We write $z_0,z_1,z_2$ for the three fixed points in $\Sigma'_1 \cup \Sigma_2'$ which project down to $p_0,p_1,p_2$ respectively. See Figure~\ref{arc_of_two_spheres}.

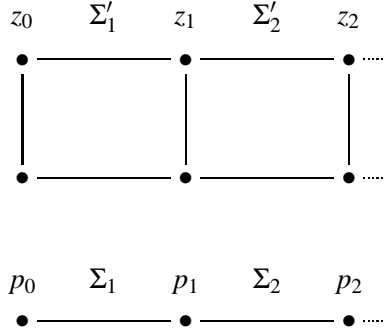
\begin{figure}
\[
\xymatrix@R=2pt@C=4pt{
z_0&&\Sigma_1'&& z_1&&\Sigma_2'&& z_2&\\
\bullet \ar@{-}[rrrr] \ar@{-}[ddddd]
	&&&& \bullet \ar@{-}[ddddd]\ar@{-}[rrrr] 
		&&&& \bullet \ar@{-}[ddddd]\ar@{.}[r] &\\
&&&& &&&& &\\
&&&& &&&& &\\
&&&& &&&& &\\
&&&& &&&& &\\
\bullet \ar@{-}[rrrr]
	&&&& \bullet \ar@{-}[rrrr] 
		&&&& \bullet \ar@{.}[r] &\\
&&&& &&&& &\\
&&&& &&&& &\\
&&&& &&&& &\\
p_0&&\Sigma_1&& p_1&&\Sigma_2&& p_2&\\
\bullet \ar@{-}[rrrr]
	&&&& \bullet \ar@{-}[rrrr] 
		&&&& \bullet \ar@{.}[r] &\\
}
\]
\caption{Configuration of fixed points in the situation of Lemma~\ref{max_at_most_two_in}.}
\label{arc_of_two_spheres}
\end{figure}

\begin{lemma}\label{max_at_most_two_in}
Suppose that neither $z_0$ nor $z_1$ is  a maximum or minimum of $H$. Then $z_2$ is either a maximum or a minimum of $H$.
\end{lemma}
\begin{proof}
The horizontal weights at $z_0$ are $(\pm 1, \pm m)$. Since $z_0$ is neither a maximum nor a minimum of $H$, Lemma~\ref{relative_weights_equal_sign} says the horizontal weights must have different signs. We assume they are $(1,-m)$ and will prove by contraction that $z_2$ is a maximum of $H$. (The case $(-1,m)$ leads in an identical way to $z_2$ being a minimum.) 

From Lemma~\ref{no_max_at_start}, we know that the horizontal weights at $z_1$ are $(m,1-m)$. This means that the horizontal weight at $z_2$ corresponding to $\Sigma_2'$ is $m-1$, which is positive. Since we are assuming $z_2$ is not a maximum, the other weight must be negative, $-n$ say. 

Now by Lemma~\ref{H_via_relative_weights}, $H(z_2) = m-1-n$ and by Lemma~\ref{horizontal_change_in_H} we have
\[
m-n-2 = (2+t)(m-1)
\]
where $t$ is the self-intersection of $\Sigma_2$. This rearranges to give that $t(m-1) + m = -n$ is negative. But this is impossible since $t \geq -1$ and $m>0$.
\end{proof}

We are now ready to limit the length of arcs of spheres in $M^\st$.

\begin{lemma}\label{no_arc_5plus}
There is no arc of five or more spheres in $M^\st$.
\end{lemma}

\begin{proof}
Suppose, for a contradiction, we have an arc of five or more spheres. Pick a connected component of the lift of the arc to $Z$. It contains at least six fixed points and so, by Lemma~\ref{max_at_most_two_in}, at least two of these points are local exteremums of $H$. Considering both connected components, we see we have found four distinct local extrema of $H$, contradicting the fact that the level sets of $H$ are connected.
\end{proof}

\begin{lemma}\label{no_arc_3}
There is no arc of three spheres in $M^\st$.
\end{lemma}
\begin{proof}
Assume for a contradiction such an arc exists. By Lemma~\ref{max_at_most_two_in}, one of the lifts of this arc to $Z$ contains the maximum of $H$ and, moreover, this maximum cannot be an endpoint. Consider the lift of the arc which contains the maximum and write the fixed points as $z_0, z_1, z_2, z_3$, in order, where $z_2$ is the maximum. Since $H\circ \gamma = -H$, we see that $\gamma(z_2)$ is the minimum of $H$

Write $(\pm 1,m)$ for the relative weights downstairs below $z_0$, where $m>1$. We will prove that the horizontal weights at $z_0$ are $(1,-m)$. Since $z_0$ is not the maximum (which is $z_2$) or the minimum (which is $\gamma(z_2)$) we know, by Lemma~\ref{relative_weights_equal_sign}, that the horizontal weights at $z_0$ have different signs and so are either $(1,-m)$ or $(-1,m)$. But if we take $(-1,m)$ then $H$ is decreasing as we move to $z_1$ then increasing as we move to $z_2$. This would imply that the horizontal weights at $z_1$ are negative, making it the minimum, which is a contradiction.

We now know, from the proof of Lemma \ref{no_max_at_start}, that the horizontal weights at $z_1$ are $(m, 1-m)$. Write $(m-1, n)$ for the horizontal weights at $z_2$, where $n>0$. (We know it is positive since $z_2$ is a maximum.) The horizontal weights at $z_3$ are then $(-n,1)$. (The second weight must have absolute value one since it lies above the end of an arc and it must be positive because $z_3$ is not a minimum). 

Applying Lemma \ref{H_via_relative_weights} we have:
\[
h(z_0) = 1-m,\quad
h(z_1) = 1, \quad
h(z_2) = m+n-1,\quad
h(z_3) = 1-n.
\]
Now write $s$ for the self-intersection of the sphere between $z_1$ and $z_2$ and $t$ for the self-intersection of the sphere between $z_2$ and $z_3$. By Lemma \ref{horizontal_change_in_H}, we have
\begin{eqnarray*}
m+n-2 &=& (2+s)(m-1)\\
m+2n - 2&=& (2+t)n
\end{eqnarray*}
These rearrange to
\[
s = \frac{n-m}{m-1},\quad\quad
t = \frac{m-2}{n}.
\]
From these equations we will find our contradiction. Note that $m-1,n$ are the orders of the stabilisers of generic points in the second and third sphere in the arc, so $n>1$ and $m>2$. It follows that $t >0$. Now $s>-2$. If $s=-1$, then $n=1$, a contradiction. Moreover, if $s \geq 0$ then $n \geq m$ from which we see that $t<1$, contradicting $m>2$. There are thus no possible solutions to our equations and inequalities, proving that no arc of length three is possible. 
\end{proof}

\begin{lemma}\label{no_arc_4}
There is no arc of 4 spheres.
\end{lemma}

\begin{proof}
Assume for a contradiction that such an arc exists. By Lemma~\ref{max_at_most_two_in}, we know that one of its lifts to $Z$ contains a maximum which, moreover, must lie at the centre of its chain. Write $z_0, z_1, z_2, z_3, z_4$ for the fixed points, in order, with $z_2$ the maximum. 

Write $m>1$ for the rotation speed of the sphere between $z_0, z_1$ and $n>1$ as the rotation speed of the sphere between $z_3$ and $z_4$. Then, from Lemma~\ref{no_max_at_start}, we have the following sequence of horizontal weights:
\[
(1,-m)\quad
(m,1-m)\quad
(m-1,n-1)\quad
(1-n, n)\quad
(-n,1)
\]
In particular, Lemma~\ref{H_via_relative_weights} gives that
\[
h(z_1) = 1 = h(z_3),\quad h(z_2) = m+n-2.
\]

Write $s$ and $t$ for the self-intersections of the projections of the spheres between $z_1,z_2$ and $z_2,z_3$ respectively. By Lemma~\ref{horizontal_change_in_H}, we have
\[
m+n-3 = (2+s)(m-1) = (2+t)(n-1).
\]
These rearrange to give 
\[
s = \frac{n-m-1}{m-1},\quad\quad
t = \frac{m-n-1}{n-1}.
\]
Note that $m,n \geq 3$, since $m-1$ and $n-1$ are rotation speeds of the second and third spheres in the chain. Adding the equations we see that
\[
s(m-1) + t(n-1) = -2
\]
So at least one of $s$ or $t$ is negative. Meanwhile, by Proposition~\ref{adjunction_inequality}, $s,t >-2$. But if $s=-1$, then $n=2$, whilst if $t=-1$ then $m=2$, either giving a contradiction.
\end{proof}

\begin{proof}[Proof of Proposition~\ref{arc_length_at_most_2}] Together, Lemmas~\ref{no_arc_5plus}, \ref{no_arc_3} and~\ref{no_arc_4} show that the longest arc of spheres possible in $M^\st$ contains two spheres. 
\end{proof}

\subsection{Circles of spheres in $M^\st$}

We next turn to circles of spheres in $M^\st$, the goal being to show that there are at most three spheres in such a circle. Each sphere in the circle has two natural lifts to $Z$ (as described in \S\ref{adjunction_section}). A~priori there are two possibilities: either the lifted spheres in $Z$ form a single connected component double-covering the circle of spheres in $M$, or they give a trivial double-cover, with two connected components. The first step is to rule out this second, disconnected, possibilty.

\begin{lemma}\label{must_be_mobius}
Given a circle of spheres in $M^\st$, the natural lifts to $Z$ form a connected set. Moreover, $H$ attains its maximum at one of the fixed points in the union of lifted spheres.
\end{lemma}

\begin{proof}
Let $C \subset Z$ denote a connected component of the lifts of the circle of spheres. Since $C$ is compact, the restriction $H|_C$ of the Hamiltonian attains a maximum on $C$. $H|_C$ is strictly monotonic along each sphere in $C$ and so attains its maximum at an $S^1$-fixed point for which both horizotnal weights are positive. By Lemma~\ref{vertical_weights}, this means the vertical weight is also positive and hence the maximum of $H|_C$ is in fact a local maximum of $H$ on the whole of $Z$. By connectedness of the level sets of $H$, this local maximum is actually the unique global maximum of $H$. By uniqueness, $C$ is the only connected component.
\end{proof}

Write $\Sigma_1, \ldots, \Sigma_n \subset M^\st$ for the spheres in the circle under consideration, written in order (so that $\Sigma_i \cap \Sigma_{i+1} \neq \varnothing$ for $i=1, \ldots n-1$). Let $p_i, p_{i+1}$ denote the fixed points in sphere $\Sigma_i$ (so that $p_{n+1} = p_1$). Choose $z_1 \in Z$ to be a fixed point lying above $p_1$. This singles out a lift $\Sigma'_1$ of $\Sigma_1$ which contains $z_1$. The other fixed point $z_2 \in \Sigma_1'$ lies above $p_2$. This again singles out a lift $\Sigma'_2$ of $\Sigma_2$ containing $z_2$, Write $z_3$ for the other fixed point of $\Sigma'_2$. Again, $z_3$ projects down to $p_3$. Note that if the circle contains only two spheres (i.e., $n=2$) then $z_1$ and $z_3$ project to the same point of $M$. The previous lemma ensures however that nonetheless $z_1$ and $z_3$ are distinct.

\begin{lemma}\label{change_of_sign}
Suppose that $H(z_1)> H(z_2) \geq 0$. Then $H(z_3) < 0$. 
\end{lemma}
\begin{proof}
Write $m_i>1$ for the size of the stabiliser of the generic point of $\Sigma_i$. The horizontal weights at $z_2$ have absolute values equal to $m_1$, $m_2$. The first step is to determine the signs of these weights. Consider the sphere $\Sigma_1'$ running from $z_1$ to $z_2$. The fact that $H(z_1) > H(z_2)$ tells us that the corresponding horizontal weight at $z_2$ is  negative, so the first weight is $-m_1$. Meanwhile, the fact that $H(z_2)\geq 0$ means that $z_2$ is not the minimum of $H$, so the other horizontal weight at $z_2$ must be positive, equal to $m_2$.  

Now Lemma~\ref{H_via_relative_weights} gives that $H(z_2) = m_2-m_1$
Meanwhile, Lemma~\ref{horizontal_change_in_H} shows $H(z_2) - H(z_3) = (2+s)m_2$, where $s$ is the self-intersection of $\Sigma_2$. (The fact that the horizontal weight at $z_2$ corresponding to $\Sigma_2'$ is positive implies that $H(z_2)>H(z_3)$.) Together these equations give that
\[
H(z_3) = - m_1 -(s+1) m_2
\]
By Propositon~\ref{adjunction_inequality}, $s \geq -1$ from which it follows that $H(z_3)<0$. 
\end{proof}

\begin{proposition}\label{circle_length_at_most_3}
A circle of spheres in $M^\st$ contains either two or three spheres.
\end{proposition}
\begin{proof}
As above, write $p_1, \ldots, p_n \in M$ for the fixed points in the circle of spheres, written in order around the circle. Assume moreover that the maximum of $H$ is attained at the point $z_1 \in Z$ which lies above $p_1$. This then determines a lift $\Sigma_1'$ of the sphere between $p_1$ and $p_2$, namely that containing $z_1$. The other fixed point on $\Sigma_1'$ gives a lift lift $z_2$ of $p_2$ and we continue in this way until, by Lemma~\ref{must_be_mobius}, we have an enumeration $z_1, z_2, \ldots, z_{2n}$ of all the fixed points in $Z$ lying above the circle where $z_{i}$ and $z_{i+n}$ both project to $p_i$.  

Now $H(z_1) > H(z_2)$, since $z_1$ is the maximum of $H$. So, by Lemma~\ref{change_of_sign}, either $H(z_2)<0$ or $H(z_3) <0$. Similarly, either $H(z_{2n})<0$ or $H(z_{2n-1})<0$. So $z_1$ lies in a sequence of at most three fixed points before the Hamiltonian becomes negative. These must be the only points at which $H$ is positive. To see this, write $C$ for the union of the lifts of the $\Sigma_i$. If there were another place where $H|_C$ were positive, there would be a second point at which $H|_C$ attained a local maximum. Just as in the proof of Lemma~\ref{must_be_mobius}, a local maximum of $H|_C$ is a global maximum of $H$ on $Z$ and so there can be only one such point.

We see then that $H$ is positive on at most three of the $z_j$ and hence, by symmetry, negative on at most three of them. It follows that there are at most six of the $z_j$ and hence $n=2$ or $3$ as claimed.
\end{proof}

Together, Propositions~\ref{arc_length_at_most_2} and~\ref{circle_length_at_most_3} prove Theorem~\ref{admiss}.

\section{All fixed points are isolated: multiple patterns}\label{isolated_end_of_proof}

Having listed all the possible patterns in $M$, we now examine when different patterns can occur simultaneously.

\subsection{Patterns bringing a maximum of $H$}
In the course of the proof of Theorem~\ref{admiss}, we saw that certain patterns bring a maximum of $H$. By uniqueness of the maximum of the Hamiltonian, we see that there is at exactly one such pattern present. We now make explicit those patterns above which $H$ does \emph{not} attain a maximum.

\begin{proposition}\label{nomax} Let $M$ be a closed 4-manifold admitting an $S^1$-invariant definite connection. If the Hamiltonian $H$ does not attain its maximum in $Z$ over a pattern in $M$, then the pattern is of one of the following two types:
\begin{enumerate}
\item  
A fixed point with relative weights  $(1,-1)$.
\item  
An arc composed of a single sphere with 
two points of relative weights $(1,-2)$. 
\end{enumerate}
It follows that $M^\st$ contains exactly one pattern from the list of Theorem~\ref{admiss} besides the two cases given here.
\end{proposition}
\begin{proof}
We have already seen that $H$ must attain its maximum above either a surface of non-isolated fixed points or a circle of spheres. It remains to treat fixed points which are isolated in $M^\st$ as well as arcs. 

We begin with fixed points which are isolated in $M^\st$. Such a point $p$ has relative weights $(a,b)$ with $|a|=1=|b|$. In the case the relative weights have the same sign, we see that one of the two fixed $z$ above $p$ has horizontal weights $(1,1)$. By Lemma~\ref{vertical_weights}, the vertical weight at this point is $2$ and, since all the weights are positive, $z$ is a maximum of $H$. 

On the other hand, when the relative weights of $p$ are $(1,-1)$, we see that any fixed point above $p$ has horizontal weights $(1,-1)$ and vertical weight $0$; the sphere above $p$ is fixed pointwise by the $S^1$-action and consists of saddle points of $H$.

We now turn to arcs. Proposition~\ref{arc_length_at_most_2} tells us that there are at most three fixed points in any arc of spheres in $M^\st$. Moreover, by Lemma~\ref{max_at_most_two_in}, three fixed points leads automatically to a maximum of $H$. So, for $H$ to not attain its maximum, the arc must in fact be a single sphere, connecting two fixed points $p_0, p_1$ with relative weights $(a_j,m)$ where $m>1$ and $|a_j|=1$. Now a fixed points $z_1$ above $p_1$ has horizontal weights equal to either$(1,-m)$ or $(-1,m)$, since if the horizontal weights had the same sign, $H$ would attain a maximum at $z_1$ or $\gamma(z_1)$. So, by Lemma~\ref{H_via_relative_weights}, $|H(z_1)| = |m-1|$. Meanwhile, by Lemma~\ref{no_max_at_start}, we see that $|H(z_1)|=1$. It follows that $m=2$ and the relative weights of $p_1$ are $(1,-2)$. The symmetric argument gives the same result for $p_0$. 
\end{proof}

\subsection{The constraint equation}

Next we bring equation~(\ref{constraint}) into play, which gives a constraint which the relative weights of the fixed points must satisfy. 

\begin{definition}
Given a pattern $P \subset M^\st$ containing only isolated fixed points, write
\[
q(P) = \sum \frac{1}{a_jb_j}
\]
where the sum is over all fixed points in $P$ and $(a_j,b_j)$ are the relative weights of these fixed poitns.
\end{definition}

Write $P_1, \ldots, P_r$ for the patterns in $M$. Then, at least when all fixed points are isolated, equation~(\ref{constraint}) says that $q(P_1) + \cdots + q(P_r) = 0$. The following two results show that the values of $q(P)$ are themselves heavily constrained. 

\begin{lemma}\label{contribution_no_max}
Let $P \subset M^\st$ be a pattern above which $H$ does \emph{not} attain its maximum. Then $q(P) = 1$. 
\end{lemma}
\begin{proof}
This is a simple calculation, using the classification of such patterns in Proposition~\ref{nomax}.
\end{proof}

\begin{lemma}\label{contribution_max}
Let $P \subset M^\st$ be a pattern above which $H$ \emph{does} attain its maximum. Then either $q(P) > -1$ or $q(P) =-1$ and $P$ is a single fixed point with relative weights $(1,1)$.
\end{lemma}
\begin{proof}
By Theorem~\ref{admiss} and Proposition~\ref{nomax}, there are three types of patter to consider: a circle of spheres, an arc of spheres and a fixed point which is isolated in $M^\st$. We treat each case separately.

\emph{Maximum occurs above a circle of spheres.} If $P$ is a circle of spheres, then by Theorem~\ref{admiss} the number of spheres is in the circle is either two or three. Suppose first that $P$ contains two fixed points with relative weights $(a_i,b_i)$ where $|a_1|=|a_2|$ and $|b_1| =|b_2|$. The smallest possible sizes of these absolute values is 2 and 3 and so $q(P) \geq-\frac{2}{2\cdot 3} = -\frac{1}{3}$. 

When $P$ is a circle of three spheres, write $m,n,p$ for the order of the stabilisers of a generic point on each sphere. Write the relative weights of the three fixed points as $(a_i,b_i)$. The three triples $(|a_i|,|b_i|)$ coincide with the three possible pairs made out of $m,n,p$. By Lemma~\ref{coprime}, $m,n,p$ are all coprime. It follows that the smallest possible values for the absolute values of the relative weights are $2,3,5$ and so $q(P) \geq-(\frac{1}{2\cdot 3}+\frac{1}{2 \cdot 5}+\frac{1}{3\cdot 5})>-1$.

\emph{Maximum occurs above an arc of spheres.} Suppose the arc contains one sphere and that the order of the stabiliser of a generic point on the sphere is at least three. Then $q(P) \geq -2\frac{1}{3}>-1$. If the arc consists of a single sphere with stabiliser of order 2, then the relative weights of the two fixed points are $(1,-2)$ and $(1,2)$ (otherwise, by Proposition~\ref{nomax}, there is no maximum of $H$ above the sphere). In this case $q(P)=0$. 

Suppose now that the arc contains two spheres. A priori, the smallest possible value of $q(P)$ occurs when the fixed points have relative weights $(1,2)$, $(2,3)$ and $(3,1)$, giving $-1$. But this gives three distinct fixed points with relative weights of the same sign, contradicting Corolloary~\ref{connected_max}. Any other combination of relative weights gives $q(P)>-1$.

\emph{Maximum occurs above a point of type $(1,1)$.} In this case $q(P) = -1$.
\end{proof}

We can now put all of these pieces together to give a complete characterisation of $M^\st$.

\begin{theorem}\label{final_list_of_patterns}
Let $M$ be a closed 4-manifold admitting an $S^1$-invariant definite connection. Suppose that all fixed points in $M$ are isolated. Then the set $M^\st$ of points with non-trivial stabiliser is one of the following:
\begin{enumerate}
\item\label{connected1}
A single circle of either two or three spheres. The relative weights of the fixed points are such that equation (\ref{constraint}) is satisfied.
\item\label{connected2}
An arc of one or two spheres. Again the relative weights of the fixed points are such that equation (\ref{constraint}) is satisfied.
\item\label{disconnected1}
A pair of points, one with relative weights $(1,1)$ the other with relative weights $(1,-1)$.
\item\label{disconnected2}
A point with relative weights $(1,1)$ and a single sphere joining two fixed points each with relative weights $(1,-2)$. 
\end{enumerate}
\end{theorem}

\begin{proof}
If the maximum of $H$ is attained above a fixed point with relative weights $(1,1)$ then this contributes $-1$ to the left-hand-side of equation~(\ref{constraint}). By Lemma~\ref{contribution_no_max} any other pattern contributes exactly $+1$ to the left-hand-side and so there must be exactly one other pattern. This accounts for cases~\ref{disconnected1} and~\ref{disconnected2} in the above list.

By Lemma~\ref{contribution_max}, if the maximum of $H$ is attained above any other pattern, $P$, then $q(P)>-1$. Since, by Lemma~\ref{contribution_no_max}, the patterns which do not bring a maximum of $H$ both have $q =1$ it follows from (\ref{constraint}) that $P$ is the only pattern present. This accounts for cases~\ref{connected1} and~\ref{connected2}.
\end{proof}

\subsection{Completion of the proof of Theorem~\ref{main_theorem}}

\begin{theorem}
Let $M$ be a closed 4-manifold admitting an $S^1$-invariant definite connection. Suppose that all fixed points in $M$ are isolated. Then $M$ is diffeomorphic to either $S^4$ or $\overline{\C\P}^2$
\end{theorem}

\begin{proof}
We first prove that $M$ is simply connected. Since the fixed point set in $M$ is a collection of points, the fixed point set in $Z$ is a collection of points and possibly some 2-spheres (fibres of $Z \to M$). It follows that the locus on which $H$ attains its maximal value is either a point or a 2-sphere, both of which are simply connected. It now follows from Theorem~\ref{huili} that $\pi_1(M)=1$. 

Looking at the list in Theorem~\ref{final_list_of_patterns} we see that there are either 2 or 3 fixed points and so $\chi(M)=2$ or $3$ accordingly. When $\chi(M)=2$ then (since $\pi_1=1$) $b_2=0$ and we see from Fintushel's Theorem~\ref{Fintushel_classification} that $M$ is diffeomorphic to $S^4$. It remains to treat the case $\chi(M)=3$. In this case $b_2=1$ and by Fintushel's Theorem~\ref{Fintushel_classification} it suffices to show that the signature is $-1$. Looking at the list in Theorem~\ref{final_list_of_patterns}, there are three possibilities with three fixed points and we divide the proof up accordingly. 

\emph{$M^\st$ is a circle of three spheres.} Cyclically order the spheres in the circle. Pick any choice of orientation on the first sphere in the circle. Orient the second so that it intersects the first positively. Now orient the third so that it intersects the second positively. The sign of the intersection between the third and first sphere is now fixed by the initial choice and agrees with the sign of the signature. (This is because the spheres are all homologous up to sign, since $b_2=1$.) 

Meanwhile, as is described in \S\ref{adjunction_section},  the choice of orientation on each sphere determines a lift of it to $Z$. Picking the orientation of the second sphere so that it meets the first positively means that the lift of the second sphere meets the lift of the first. Similarly the lift of the  third sphere meets the lift of the second. Now if the signature were $1$, the third and first sphere would intersect positively, their lifts would meet and form a loop of three spheres. Reversing all the orientations would give a second connected component of lifts contradicting Lemma~\ref{must_be_mobius}. We see then that the signature must be $-1$ as claimed.

\emph{$M^\st$ is an arc of two spheres.} Write the relative weights at the end-points of the arc as $(1,a)$ and $(b,1)$. The middle fixed point has relative weights $(p,q)$ with $|p|=|a|$ and $|q|=|b|$. There are thus two possibilities: either the relative weights at the middle point are $(a,b)$ or they are $(-a,b)$. 

Suppose first that the middle point $p$ has relative weights $(-a,b)$. The constraint equation~(\ref{constraint}) implies $b = 1-a$. Now the relative weights at $p$ have the same sign. This means that at a fixed point $z \in Z$ above $p$, the horizontal weights also have the same sign. So the maximum occurs above $p$. Lemma~\ref{no_max_at_start} then gives that $|H(z)| = 1$ so $|1-2a| =1$. The only way this can happen is either $a=0$ or $a=1$ both of which contradict the fact $a$ is a relative weight at the start of the arc.

Next suppose that the relative weights at $p$ are $(a,b)$. The constraint equation~(\ref{constraint}) now implies $b = -1 - a$. We now compute the signature directly from equation~(\ref{signature}). The end result is $-1$ as claimed. 

\emph{$M^{S^1}$ is one point with relative weights $(1,1)$ and a pair of points each with relative weights $(1,-2)$.} In this case we can compute the signature directly from equation~(\ref{signature}), which gives the answer $-1$ as claimed.
\end{proof}

\section{Open questions and further directions}\label{conclusion}

Theorem \ref{main_theorem} is partial case of the following conjecture,
that first appeared in \cite{Fine2009Symplectic-Cala} and which remains completely open.

\begin{conjecture} If a closed 4-manifold $M$ admits a positive definite connection then $M$ is diffeomorphic to either $S^4$ or $\overline{\C\P}^2$.
\end{conjecture}

This conjecture is discussed further in \cite{Fine2014A-gauge-theoret} in the general (non-equivariant) setting, along with its relation to Einstein metrics.

Theorem \ref{main_theorem}  should admit series of generalisations in several directions. First of all, we expect that the methods of this article will also lead to a classification of diffeomorphism types of four-dimensional orbifolds with $S^1$-invariant definite connections. We hope to address this in future work. At the same time, even in the case of manifolds our methods do not say anything about the nature of the $S^1$-action. There are many exotic $S^1$-actions on $S^4$ or  $\overline{\C\P}^2$. For these actions the set of points with non-trivial stabiliser projects to a non-trivial knot in the quotient $M/S^1$. By analogy with the result of Grove and Wilking (the second part of Theorem \ref{HKGW}) we expect that the answer to the following question is positive.  

\begin{question} Suppose we have an $S^1$-invariant definite connection on 
$S^4$ or $\overline{\C\P}^2$. Is it true that the $S^1$-action is conjugate to a standard linear action?
\end{question}

All these questions have their Riemannian counterparts on Riemanninan 
manifolds and orbifolds satisfying curvature inequality of Corollary~\ref{Riemannian_cor}.

In a different direction, it would be interesting classify or at least find bounds on those six-dimensional symplectic Fano manifolds admitting Hamiltonian $S^1$-actions. We will express this in the following conjecture.

\begin{conjecture}\label{S1_integrable_Fano_question}
Let $(Z,\omega)$ be a six-dimensional symplectic Fano manifold with a Hamiltonian $S^1$-action. Is it true that there is an integrable $S^1$-invariant complex structure on $Z$ compatible with~$\omega$?
\end{conjecture}

Dusa McDuff communicated to us that in the case when $H^2(Z,\mathbb R)=\mathbb R$ the methods of her work \cite{McDuff2009Some-6-dimensio} give a proof of this conjecture. Yunhyung Cho has told us that in recent work which will appear soon, he proves the conjecture in the case when the $S^1$-action on $Z$ is semi-free.

Finally, it would be interesting to extend our approach  to get a topological classification of some $S^1$-invariant fat bundles of higher dimension. Originally fat bundles were introduced by Weinstein, motivated by the search for new manifolds of positive sectional curvature. The fatness condition arises naturally when studying such metrics on the total space of a fibration with totally geodesic fibres. There has been considerable work in this direction, which is surveyed in \cite{Ziller2000Fatness-revisit} and \cite[Section 4]{Ziller2014Riemannian-mani}. On the other hand, the symplectic interpretation of fat bundles, has not been explored to such an extent. The symplectic point of view was described by Weinstein in \cite{Weinstein1980Fat-bundles-and} and is encoded in the notion of $S$-fatness which we now recall. 

Let $P\to M$ be a principal $G$-bundle over $M$ with a connection $A$ and denote by $\g$ the Lie algebra of $G$. The curvature $F_A$ of this bundle is a section of $\g\otimes \Lambda^2(T^*M)$. So, interpreting $F_A$ as a section of $\Hom(\g^*,\Lambda^2(T^*M))$, for any $\mu\in \g^*$ we get a section $F_A(\mu)\in \Lambda^2(T^*M)$. 

\begin{definition}\cite{Weinstein1980Fat-bundles-and}
Let $S\subset \g^*$ be a union of non-zero coadjoint orbits in $\g^*$. We say that the connection $A$ is \emph{$S$-fat} if for any $\mu\in S$ the section $F_A(\mu)\in \Lambda^2(T^*M)$ is a nowhere degenerate two-form on $M$. In the case $S=\g^*\setminus 0$ we say simply that $A$ is \emph{fat}.
\end{definition}

Let $\mathcal{O}$ be a coadjoint orbit of $G$ and suppose and a principal $G$-bundle $E$ over $M$ is $\mathcal{O}$-fat. In this case Weinstien explained \cite[Theorem 3.2]{Weinstein1980Fat-bundles-and} that the associated $\mathcal{O}$-bundle over $M$ is symplectic. Reznikov discovered examples of such types of bundles originating from Riemannian geometry. He observed \cite{Reznikov1993Symplectic-twis} that the twistor space of an even dimensional Rimannian manifold $M^{2n}$, satisfying a certain curvature inequality is naturally symplectic. The inequality can be paraphrased as saying the curvature is either ``sufficiently positive'' or ''sufficiently negative''. In terms of fatness, here the principal bundle is the $\SO(2n)$-frame bundle with induced Levi-Civita connection, and the coadjoint orbit $\mathcal{O}$ is the space of anti-symmetric $2n\times 2n$ matrices with square $-1$. Reznikov's condition on curvature is equivalent to $\mathcal{O}$-fatness of $\SO(2n)$-bundle. Thus one can hope to study sufficiently positively curved even-dimensional manifolds with an isometric $S^1$-action using methods of symplectic geometry as in  Corollary  \ref{Riemannian_cor}. The case of dimension six seem to be the most promising. We hope to address this in future work.

\bibliographystyle{abbrv}
\bibliography{S1_bibliography}

\begin{thebibliography}{10}

\bibitem{Atiyah1968The-index-of-el}
M.~F. Atiyah and I.~M. Singer.
\newblock The index of elliptic operators. {III}.
\newblock {\em Ann. of Math. (2)}, 87:546--604, 1968.

\bibitem{Bott1967Vector-fields-a}
R.~Bott.
\newblock Vector fields and characteristic numbers.
\newblock {\em Michigan Math. J.}, 14:231--244, 1967.

\bibitem{Cho2012Chern-classes-a}
Y.~Cho, M.~K. Kim, and D.~Y. Suh.
\newblock Chern classes and symplectic circle actions.
\newblock {\em arXiv preprint arXiv:1207.4977}, 2012.

\bibitem{Derdzinski1981Unflat-connecti}
A.~Derdzinski and A.~Rigas.
\newblock Unflat connections in 3-sphere bundles over ${S}^4$.
\newblock {\em Transactions of the American Mathematical Society},
  265(2):485--493, 1981.

\bibitem{Fine2009Symplectic-Cala}
J.~Fine and D.~Panov.
\newblock Symplectic {C}alabi-{Y}au manifolds, minimal surfaces and the
  hyperbolic geometry of the conifold.
\newblock {\em J. Differential Geom.}, 82(1):155--205, 2009.

\bibitem{Fine2010Hyperbolic-geom}
J.~Fine and D.~Panov.
\newblock Hyperbolic geometry and non-{K}\"ahler manifolds with trivial
  canonical bundle.
\newblock {\em Geometry and Topology}, 14(3):1723--1764, 2010.

\bibitem{Fine2013The-diversity-o}
J.~Fine and D.~Panov.
\newblock The diversity of symplectic {C}alabi-{Y}au 6-manifolds.
\newblock {\em J. Topol.}, 6(3):644--658, 2013.

\bibitem{Fine2014A-gauge-theoret}
J.~Fine, D.~Panov, and K.~Krasnov.
\newblock A gauge theoretic approach to {E}instein 4-manifolds.
\newblock {\em New York J. Math.}, 20:293--324, 2014.

\bibitem{Fintushel1977Circle-actions-}
R.~Fintushel.
\newblock Circle actions on simply connected {$4$}-manifolds.
\newblock {\em Trans. Amer. Math. Soc.}, 230:147--171, 1977.

\bibitem{Fintushel1978Classification-}
R.~Fintushel.
\newblock Classification of circle actions on {$4$}-manifolds.
\newblock {\em Trans. Amer. Math. Soc.}, 242:377--390, 1978.

\bibitem{Gromov1985Pseudo-holomorp}
M.~Gromov.
\newblock Pseudo holomorphic curves in symplectic manifolds.
\newblock {\em Invent. Math.}, 82(2):307--347, 1985.

\bibitem{Grove2013A-knot-characte}
K.~Grove and B.~Wilking.
\newblock A knot characterization and 1-connected nonnegatively curved
  4-manifolds with circle symmetry.
\newblock {\em arXiv preprint arXiv:1304.4827}, 2013.

\bibitem{Hirzebruch1971The-signature-t}
F.~Hirzebruch.
\newblock The signature theorem: reminiscences and recreation.
\newblock In {\em Prospects in mathematics ({P}roc. {S}ympos., {P}rinceton
  {U}niv., {P}rinceton, {N}.{J}., 1970)}, pages 3--31. Ann. of Math. Studies,
  No. 70. Princeton Univ. Press, Princeton, N.J., 1971.

\bibitem{Hsiang1989On-the-topology}
W.-Y. Hsiang and B.~Kleiner.
\newblock On the topology of positively curved {$4$}-manifolds with symmetry.
\newblock {\em J. Differential Geom.}, 29(3):615--621, 1989.

\bibitem{Li2006The-fundamental}
H.~Li.
\newblock The fundamental group of symplectic manifolds with {H}amiltonian
  {L}ie group actions.
\newblock {\em J. Symplectic Geom.}, 4(3):345--372, 2006.

\bibitem{McDuff1990The-structure-o}
D.~McDuff.
\newblock The structure of rational and ruled symplectic {$4$}-manifolds.
\newblock {\em J. Amer. Math. Soc.}, 3(3):679--712, 1990.

\bibitem{McDuff2009Some-6-dimensio}
D.~McDuff.
\newblock Some 6-dimensional {H}amiltonian {$S^1$}-manifolds.
\newblock {\em J. Topol.}, 2(3):589--623, 2009.

\bibitem{McDuff1998Introduction-to}
D.~McDuff and D.~Salamon.
\newblock {\em Introduction to symplectic topology}.
\newblock Oxford Mathematical Monographs. The Clarendon Press, Oxford
  University Press, New York, second edition, 1998.

\bibitem{Pao1978Nonlinear-circl}
P.~S. Pao.
\newblock Nonlinear circle actions on the {$4$}-sphere and twisting spun knots.
\newblock {\em Topology}, 17(3):291--296, 1978.

\bibitem{Reznikov1993Symplectic-twis}
A.~G. Reznikov.
\newblock Symplectic twistor spaces.
\newblock {\em Ann. Global Anal. Geom.}, 11(2):109--118, 1993.

\bibitem{Searle1994On-the-topology}
C.~Searle and D.~Yang.
\newblock On the topology of non-negatively curved simply connected
  {$4$}-manifolds with continuous symmetry.
\newblock {\em Duke Math. J.}, 74(2):547--556, 1994.

\bibitem{Taubes2000Seiberg--Witten}
C.~H. Taubes.
\newblock {\em {S}eiberg--{W}itten and {G}romov invariants for symplectic
  4-manifolds}.
\newblock International Press Somerville, 2000.

\bibitem{Tolman2010On-a-symplectic}
S.~Tolman.
\newblock On a symplectic generalization of {P}etrie's conjecture.
\newblock {\em Trans. Amer. Math. Soc.}, 362(8):3963--3996, 2010.

\bibitem{Weinstein1968Unflat-bundles}
A.~Weinstein.
\newblock Unflat bundles (premilnary report).
\newblock Technical report, University of California, Berkeley, 1968.

\bibitem{Weinstein1980Fat-bundles-and}
A.~Weinstein.
\newblock Fat bundles and symplectic manifolds.
\newblock {\em Adv. in Math.}, 37(3):239--250, 1980.

\bibitem{Wright2011Compact-anti-se}
D.~Wright.
\newblock Compact anti-self-dual orbifolds with torus actions.
\newblock {\em Selecta Math. (N.S.)}, 17(2):223--280, 2011.

\bibitem{Ziller2000Fatness-revisit}
W.~Ziller.
\newblock Fatness revisited.
\newblock Preprint, 2000.

\bibitem{Ziller2014Riemannian-mani}
W.~Ziller.
\newblock {\em Riemannian manifolds with positive sectional curvature}, volume
  2110 of {\em Lecture Notes in Mathematics}.
\newblock Springer, 2014.

\end{thebibliography}

\noindent{\scshape Joel Fine}\\
{\tt joel.fine@ulb.ac.be}\\
Département de mathématique, Université libre de Bruxelles, Bruxelles 1050, Belgium.
\vspace{\baselineskip}

\noindent{\scshape Dmitri Panov}\\
{\tt dmitri.panov@kcl.ac.uk}\\
Department of Mathematics, King's College, Strand, London WC2R 2LS, United Kingdom.

\end{document}